\documentclass[11pt]{amsart}
\usepackage{geometry}
\usepackage{graphicx}
\usepackage{amssymb}
\usepackage{amsmath, mathrsfs, stmaryrd}
\usepackage{color}
\newtheorem{theorem}{Theorem}
\newtheorem{definition}[theorem]{Definition}
\newtheorem{lemma}[theorem]{Lemma}
\newtheorem{corollary}[theorem]{Corollary}

\newcommand{\R}{\mathbb R}
\newcommand{\HH}{\mathbb H}
\numberwithin{theorem}{section}
\numberwithin{equation}{section}
\title{The Minkowski formula and the quasi-local mass}
\author{Po-Ning Chen, Mu-Tao Wang, and Shing-Tung Yau}
\date{}
\thanks{P.-N. Chen is supported by NSF grant DMS-1308164, M.-T. Wang is supported by NSF grants DMS-1105483 and DMS-1405152,  and S.-T. Yau is supported by NSF
grants  PHY-0714648 and DMS-1308244. This work was partially supported by a grant from the Simons Foundation (\#305519 to Mu-Tao Wang). Part of this work was carried out
when P.-N. Chen and M.-T. Wang were visiting the Department of Mathematics and the Center of Mathematical Sciences and Applications at Harvard  University.}

\begin{document}
\begin{abstract}
In this article, we estimate the quasi-local energy with reference to the Minkowski spacetime \cite{Wang-Yau1,Wang-Yau2}, the anti-de Sitter spacetime \cite{Chen-Wang-Yau}, or the Schwarzschild spacetime \cite{Chen-Wang-Wang-Yau}.  In each case, the reference spacetime admits a conformal Killing--Yano 2-form which facilitates the application of  the Minkowski formula in \cite{Wang-Wang-Zhang} to estimate the quasi-local energy. As a consequence of the positive mass theorems  in \cite{Liu-Yau,Shi-Tam} and the above estimate, we obtain rigidity theorems which characterize the Minkowski spacetime and the hyperbolic space.
\end{abstract}
\maketitle

\section{Introduction}
In this article, we estimate the quasi-local mass with reference to the Minkowski spacetime \cite{Wang-Yau1,Wang-Yau2}, the anti-de Sitter spacetime \cite{Chen-Wang-Yau}, or the Schwarzschild spacetime \cite{Chen-Wang-Wang-Yau}.  In each case, the reference spacetime admits a conformal Killing--Yano 2-form. As a result of this ``hidden symmetry", the classical Minkowski formula is extended to spacelike codimension-two submanifolds in these reference spacetimes by Wang, Wang and Zhang in \cite{Wang-Wang-Zhang}. Our estimate of the quasi-local energy is based on the $k=2$ Minkowski formula for surfaces in these reference spacetimes.

In the classical Minkowski formula \cite{Minkowski}, the mean curvature $H$ and the Gauss curvature $K$ of a surface $\Sigma$ in $\R^3$ are related as follows:
\[ \int H d \Sigma = \int K (X \cdot e_3) d \Sigma,  \]
where $X$ is the position vector of $\R^3$ and $e_3$ is the outward unit normal of the surface $\Sigma$. A major application of the classical Minkowski formula is the rigidity of isometric convex surfaces in $\R^3$, namely, two convex surfaces  $\Sigma_1$ and $\Sigma_2$ in $\R^3$ with the same induced metric are the same up to an isometry of $\R^3$ \cite{Cohn-Vossen}.  The Minkowski formula is used to evaluate the integral of the difference of the mean curvatures.

Let $H_1$ and $H_2$  be the mean curvatures, and let $h_1$ and $h_2$ be the second fundamental forms of the surfaces  $\Sigma_1$ and $\Sigma_2$ given by the two embeddings $X_1, X_2$ into $\R^3$, respectively. It follows from the Minkowski formula that
\[
\begin{split}
\int (H_1 - H_2)  d \Sigma  = &  \int det (h_1 -h_2 )(X_1 \cdot e_3) d \Sigma.\\
\end{split}
\] 
Using the Gauss equations of the surfaces, we can conclude that 
\[det(h_1-h_2) \le 0.\] Reversing the order of two surfaces, we obtain $h_1=h_2$ and thus the two embeddings differ by an isometry of $\R^3$. 
In this article, we use the Minkowski formula in a similar manner to evaluate the quasi-local energy. For the Minkowski reference, this is done in Theorem \ref{eq_comparison_quasilocal_mass}. For the anti-de Sitter reference, this is done in Theorem \ref{comparison_quasilocal_mass_ads} and for the Schwarzschild reference, this is done in Theorem \ref{comparison_quasilocal_mass_sch}.

Using Theorem \ref{eq_comparison_quasilocal_mass} and Theorem \ref{comparison_quasilocal_mass_ads}, we derive upper bounds of the quasi-local energy in terms of the curvature tensor of the physical spacetime. See Theorem \ref{upper_bound_curvature_minkowski} and Corollary \ref{upper_bound_Brown-York} for the Minkowski reference and Corollary \ref{Quasi_local_mass_as_curvature_ads} for the anti-de Sitter reference. Combining the upper bound of the quasi-local energy with the positive mass theorems \cite{Liu-Yau,Shi-Tam}, we obtain rigidity theorems which characterize the Minkowski space, the Euclidean 3-space and the hyperbolic space. See Theorem \ref{rigidity_Minkowski}, Corollary \ref{rigidity_R3} and Theorem \ref{rigidity_ads}, respectively.
\section{Killing--Yano 2-form and the Minkowski formula}
In this section, we review the Minkowski formula of surfaces in 4-dimensional spacetimes admitting a Killing--Yano 2-form \cite{Wang-Wang-Zhang}.
First we recall the definition of the Killing--Yano 2-form. 
\begin{definition}
Let $Q$ be a 2-form in an (n+1)-dimensional pseudo-Riemannian manifold $(\mathfrak N ,\langle ,\rangle)$ with Levi-Civita connection $D$. $Q$ is called a conformal Killing--Yano 2-form if 
\begin{equation}
D_{X}Q(Y,Z) +D_{Y}Q(X,Z) = \frac{1}{n}\left (2 \langle X,Y\rangle  \langle \xi, Z\rangle - \langle X,Z\rangle  \langle \xi, Y\rangle  - \langle Y,Z\rangle  \langle \xi, X\rangle\right  ),
\end{equation}
where $\xi = div Q$.
\end{definition}
In \cite{Wang-Wang-Zhang}, the Minkowski formulae for higher order mixed mean curvatures are derived for submanifolds in a spacetime admitting a Killing--Yano 2-form. In particular, for a spherical symmetric spacetime $\mathfrak N$ with metric
\[  -f(r)^2 dt^2 + \frac{dr^2}{f^2(r)}+r^2 dS^2,  \]
the 2-form \[Q= r dr \wedge dt\] is a conformal Killing--Yano 2-form wtih
\[ div Q = -3 \frac{\partial}{\partial t}.\] 
Let $\Sigma$ be a 2-surface in $\mathfrak N$. Let $\{ e_3, e_4\}$ be a frame of the normal bundle of $\Sigma$ in $\mathfrak N$.  Let $h_3$ and $h_4$ be the second fundamental form of $\Sigma$ in $\mathfrak N$ in the direction of $e_3$ and $e_4$, respectively and let
\[ \alpha_{e_3} (\cdot)= \langle D_{(\cdot)} e_3, e_4 \rangle  \]
be the connection 1-form of the normal bundle determined by the frame $\{ e_3, e_4 \}$. 

Let $H_0$ be the mean curvature vector of $\Sigma$ in $\mathfrak N$ and $J_0$ be the reflection of $H_0$ through the light cone in the normal bundle of $\Sigma$. The $(r,s)=(2,0)$ Minkowski formula \cite[Theorem 4.3]{Wang-Wang-Zhang} is 
\[\begin{split}
&- \int _{\Sigma} \langle J_0 , \frac{\partial }{\partial t} \rangle  d \Sigma\\
=& \int _{\Sigma} \Big\{2 ( det(h_3) - det(h_4)) Q_{34} + ({R^{ab}}_{a3}Q_{b4} -{R^{ab}}_{a4}Q_{b3})  +[ {R^{ab}}_{43} - (d\alpha_{e_3})^{ab}]Q_{ab} \Big\} d \Sigma,
\end{split}
\]
where $R$ denote the curvature tensor of the spacetime $\mathfrak N$. This generalized the $k=2$ Minkowski formula for surfaces in $\R^3$.

In particular, for surfaces in the Minkowski space, the curvature tensor vanish and the formula reduces to 
\[
- \int _{\Sigma} \langle J_0 , \frac{\partial }{\partial t} \rangle  d \Sigma= \int _{\Sigma} \{ 2 ( det(h_3) - det(h_4)) Q(e_3,e_4 )  - (d\alpha_{e_3})^{ab}Q_{ab}\} d \Sigma.
\]
\section{The Minkowski identity and the quasi-local energy}
In this section, we rewrite the Wang--Yau quasi-local energy \cite{Wang-Yau1,Wang-Yau2} using the Minkowski formula for surfaces in $\R^{3,1}$.

Let $N$ be our physical spacetime. Given a spacelike 2-surface $\Sigma$ in $N$, let $\{ e'_3, e'_4\}$ be a frame of the normal bundle of $\Sigma$ in $N$.  Let $h'_3$ and $h'_4$ be the second fundamental forms of $\Sigma$ in $N$ in the direction of $e'_3$ and $e'_4$, respectively, and let $\alpha_{e'_3} $ be the connection 1-form
\[ \alpha_{e'_3} (\cdot)= \langle \nabla_{(\cdot)}^N e'_3, e'_4 \rangle. \]

Let $X$ be an isometric embedding of $\Sigma$ into $\R^{3,1}$. Let  $\{ e_3, e_4\}$ be a frame of the normal bundle of $X(\Sigma)$ in $\R^{3,1}$ and  $h_3$, $h_4$ and  $\alpha_{e_3}$ be the corresponding second fundamental forms and the connection 1-form.

\begin{theorem} \label{Minkowski_compare}
Given a  spacelike 2-surface $\Sigma$  in $N$ and a frame $\{ e'_3, e'_4\}$ of the normal bundle, let $X$ be an isometric embedding of $\Sigma$ into $\R^{3,1}$. Suppose there is a frame $\{ e_3, e_4\}$ of the normal bundle of $X(\Sigma)$ in $\R^{3,1}$ such that
\[   \alpha_{e'_3}=\alpha_{e_3}.\]
Then we have
\begin{equation} \label{eq_comparison_minkowski_2}
\begin{split}
   & \int \{-\langle \frac{\partial}{\partial t} , e_4 \rangle(tr h_3 -tr h'_3) +\langle \frac{\partial}{\partial t} , e_3 \rangle(tr h_4 -tr h'_4) \} d  \Sigma \\
=&  \int [ 2det(h_3) -2det(h_4) - tr h_3 tr h_3' +h_3 \cdot h_3' +tr h_4 tr h_4' - h_4 \cdot h_4' ] Q_{34}d  \Sigma \\
 &+\int  \{ R^{ab}_{\;\;\;\; a4} Q_{b3}-R^{ab}_{\;\;\;\; a3}Q_{b4} - Q_{bc} \sigma^{cd}[(h_{3})_{da}  h_4^{'ab}-(h_{4})_{da}  h_3^{'ab}  ] \}d  \Sigma, \\
\end{split}  
\end{equation}
where $R$ is the curvature tensor for the spacetime $N$.
\end{theorem}
\begin{proof}
We consider the following two divergence quantities on $\Sigma$:
\[
\nabla_a \left( [ tr h_3\sigma^{ab} - h_3^{ab}]Q_{b4} - [ tr (h_{4})\sigma^{ab} - h_{4}^{ab}]Q_{b3}\right)
\]
and 
\[
\nabla_a \left( [ tr h'_3\sigma^{ab} - h_3^{'ab}]Q_{b4} - [ tr (h'_{4})\sigma^{ab} - h_{4}^{'ab}]Q_{b3}\right).
\]
The first divergence quantitiy is exactly the one considered in \cite[Theorem 4.3]{Wang-Wang-Zhang}. It gives
\begin{equation}  \label{divergence_one_int}
    \int \{ -\langle \frac{\partial}{\partial t} , e_4 \rangle(tr h_3) +\langle \frac{\partial}{\partial t} , e_3 \rangle(tr h_4) \} d  \Sigma = \int [2[ det(h_3) -det(h_4)]   -(d\zeta)^{ab}Q_{ab} ]d  \Sigma.
\end{equation}
where $\zeta = \alpha_{e'_3}=\alpha_{e_3}$.

For the second divergence quantity, following the proof of \cite[Theorem 4.3]{Wang-Wang-Zhang} (see also \cite[Theorem 3.3]{Wang}), we compute
\begin{equation}
\begin{split}\label{1-Minkowski formula: Codazzi equation}
\nabla_a (tr(h'_3)\sigma^{ab} - h_3^{'ab}) &= -{{R}^{ab}}_{a3} -\zeta^b tr h'_{4} +\zeta_a h_{4}^{'ab} \\
\nabla_a (tr(h'_{4})\sigma^{ab} - h_{4}^{'ab}) &= -{R^{ab}}_{a4} - \zeta^b trh'_3 + \zeta_a h_3^{'ab}.
\end{split}
\end{equation}
On the other hand,
\begin{equation}
\begin{split}\label{1-Minkowski formula: nabla Q}
\nabla_a Q_{b4} &= (D_a Q)_{b4} - (h_3)_{ab}Q_{34} + Q_{bc} \sigma^{cd}(h_{4})_{da} - Q_{b3}\zeta_a \\
\nabla_a Q_{b3} &= (D_a Q)_{b3}+( h_4)_{ab}Q_{43} + Q_{bc} \sigma^{cd}(h_{3})_{da}- Q_{b4}\zeta_a.
\end{split}
\end{equation}
Putting (\ref{1-Minkowski formula: Codazzi equation}) and (\ref{1-Minkowski formula: nabla Q}) together, we get
\begin{equation}\label{after applying Codazzi and nabla Q}
\begin{split}
   &\nabla_a \left( (tr(h'_3) \sigma^{ab} - h_3^{'ab})Q_{b4} - (tr(h'_{4})\sigma^{ab} - h_{4}^{'ab})Q_{b3}\right) \\
= & {R^{ab}}_{a3}Q_{b4} -{R^{ab}}_{a4}Q_{b3} +  (tr(h'_3) \sigma^{ab} - h_3^{'ab})\left(  (D_a Q)_{b4} - (h_3)_{ab}Q_{34} + Q_{bc} \sigma^{cd}(h_{4})_{da}  \right) \\
&-(tr(h'_{4})\sigma^{ab} - h_{4}^{'ab}) \left(  (D_a Q)_{b3}+( h_4)_{ab}Q_{43} + Q_{bc} \sigma^{cd}(h_{3})_{da}\right).
\end{split}
\end{equation}
From the definition of conformal Killing-Yano 2-forms, we have
\begin{align}
  (tr(h'_3) \sigma^{ab} - h_3^{'ab})(D_a Q)_{b4}  \notag 
=& \frac{1}{2} (tr(h'_3) \sigma^{ab} - h_3^{'ab})((D_a Q)_{b4} + (D_b Q)_{a4})\\
 =&   \langle \frac{\partial}{\partial t}, tr(h'_3) e_{4}\rangle. \notag
\end{align}
Similarly,
\begin{align*}
  (tr(h'_4) \sigma^{ab} - h_4^{'ab})(D_a Q)_{b3} =  \langle \frac{\partial}{\partial t}, tr(h'_4) e_{3}\rangle. \notag
\end{align*}
Collecting terms, we get
\begin{equation} \label{divergence_two}\begin{split}
&\nabla_a \left( [ tr h'_3\sigma^{ab} - h_3^{'ab}]Q_{b4} - [ tr (h'_{4})\sigma^{ab} - h_{4}^{'ab}]Q_{b3}\right)\\
 = &R^{ab}_{\;\;\;\; a3}Q_{b4}- R^{ab}_{\;\;\;\; a4} Q_{b3} + \langle \frac{\partial}{\partial t} , e_4 \rangle( tr h'_3) -\langle \frac{\partial}{\partial t} , e_3 \rangle( tr h'_4)-(d\zeta)^{ab}Q_{ab}\\
&+ [tr h_3 tr h_3' -h_3 \cdot h_3' -tr h_4 tr h_4' + h_4 \cdot h_4' ] Q_{34}+  Q_{bc} \sigma^{cd}[(h_{3})_{da}  h_4^{'ab}-(h_{4})_{da}  h_3^{'ab} ].
\end{split}
\end{equation}
Integrating \eqref{divergence_two} over $\Sigma$, we get

\begin{equation} \label{divergence_two_int}
\begin{split}
   &\int \{ - \langle \frac{\partial}{\partial t} , e_4 \rangle( tr h'_3) + \langle \frac{\partial}{\partial t} , e_3 \rangle( tr h'_4) \}  d\Sigma \\
= & \int  [tr h_3 tr h_3' -h_3 \cdot h_3' -tr h_4 tr h_4' + h_4 \cdot h_4' ] Q_{34}  d\Sigma \\
  & +\int \{ R^{ab}_{\;\;\;\; a3}Q_{b4}- R^{ab}_{\;\;\;\; a4} Q_{b3}   +  Q_{bc} \sigma^{cd}[(h_{3})_{da}  h_4^{'ab}-(h_{4})_{da}  h_3^{'ab} ] \}d \Sigma.
\end{split}
\end{equation}  
\eqref{eq_comparison_minkowski_2} follows from the difference between  \eqref{divergence_one_int} and  \eqref{divergence_two_int}.
\end{proof}
Next, we relate the left hand side of \eqref{eq_comparison_minkowski_2} to the Wang-Yau quasi-local energy when the frame $\{ e'_3, e'_4\}$  and $\{ e_3, e_4\}$ are the canonical gauge corresponding to a pair of an isometric embedding $X$ of  $\Sigma$ into $\R^{3,1}$ and a constant future directed timelike unit vector $T_0$. 

Let $\widehat \Sigma$ be the projection of $X(\Sigma)$ onto the orthogonal complement of $T_0$. Let $\breve e_3$ be the unit outward normal of $\widehat \Sigma$ in the  orthogonal complement of $T_0$. We extend $\breve e_3$ along $T_0$ by parallel translation. Let $\breve e_4$ be the unit normal of $\Sigma$ which is also normal to  $\breve e_3$. Let $\{ \bar e_3, \bar e_4\}$ be the unique frame of the normal bundle of $\Sigma$ in $N$ such that 
\[  \langle H ,  \bar e_4 \rangle =  \langle H_0 ,  \breve e_4 \rangle. \]
The quasi-local energy of $\Sigma$ with respect to the pair $(X,T_0)$ is 
\[  E(\Sigma,X,T_0)= \frac{1}{8 \pi} \int [-  \langle H_0 ,  \breve e_3 \rangle  \sqrt{1+|\nabla \tau |^2}- \alpha_{\breve e_3}(\nabla \tau)+\langle H ,  \bar e_3 \rangle \sqrt{1+|\nabla \tau |^2} + \alpha_{\bar e_3}(\nabla \tau)] d\Sigma. \]
Let $h_3$, $h_3'$, $h_4$ and $h_4'$ be the second fundamental form in the directions of $\breve e_3$, $\bar e_3$, $\breve e_4$ and $\bar e_4$, respectively.
If $ \alpha_{\bar e_3}= \alpha_{\breve e_3}$, we have
\begin{equation} \label{eq_comparison_mass_2}
\begin{split}
   &  \frac{1}{8 \pi}  \int [-\langle \frac{\partial}{\partial t} , e_4 \rangle(tr h_3 -tr h'_3) +\langle \frac{\partial}{\partial t} , e_3 \rangle(tr h_4 -tr h'_4) ] d  \Sigma \\
=&   \frac{1}{8 \pi}  \int   [-  \langle H_0 ,  \breve e_3 \rangle  \sqrt{1+|\nabla \tau |^2}+\langle H ,  \bar e_3 \rangle \sqrt{1+|\nabla \tau |^2}] d\Sigma \\
 =&E(\Sigma,X,T_0)
\end{split}  
\end{equation}
Hence, we have proved the following:
\begin{theorem} \label{eq_comparison_quasilocal_mass}
Given a surface $\Sigma$ in the spacetime $N$, suppose we have a pair  $(X,T_0)$ of observer such that $ \alpha_{\bar e_3}= \alpha_{\breve e_3}$. Then we have
\begin{equation} 
\begin{split}
   E(\Sigma,X,T_0) 
=& \frac{1}{8 \pi} \int \{ 2[ det(h_3) -det(h_4)]  - [tr h_3 tr h_3' -h_3 \cdot h_3' -tr h_4 tr h_4' + h_4 \cdot h_4' ] Q_{34} \} d  \Sigma \\
 & + \frac{1}{8 \pi} \int   \{ R^{ab}_{\;\;\;\; a4} Q_{b3}-R^{ab}_{\;\;\;\; a3}Q_{b4} - Q_{bc} \sigma^{cd}[(h_{3})_{da}  h_4^{'ab}-(h_{4})_{da}  h_3^{'ab} ] \} d  \Sigma \\
\end{split}  
\end{equation}
\end{theorem}
\section{Upper bound of the Liu-Yau quasi-local mass}
In this section, we apply Theorem \ref{eq_comparison_quasilocal_mass} to the Liu-Yau  quasi-local mass.  For a surface $\Sigma$ in $M$, we consider the isometric embedding $X$ of $\Sigma$ into the orthogonal complement of $T_0=\frac{\partial}{\partial t}$. The quasi-local energy $E(\Sigma,X,T_0) $ is precisely the Liu-Yau quasi-local mass $m_{LY}(\Sigma)$ of the surface. To apply Theorem \ref{eq_comparison_quasilocal_mass}, we assume that $\alpha_H = 0$. Under the assumption, $m_{LY}(\Sigma)$  is a critical point of the Wang-Yau quasi-local energy. We have the following lemma concerning the Liu-Yau mass:

\begin{lemma} \label{Quasi_local_mass_as_curvature}
Suppose $\alpha_H = 0$.  The Liu-Yau mass is
\begin{equation}  \label{Liu-Yau-Original}
   m_{LY}(\Sigma) = \frac{1}{8 \pi} \int \{[2 det(h_3) - tr h_3 tr h_3' + h_3 \cdot h_3'  ]  (X \cdot e_3)-R^{ab}_{\;\;\;\; a3} (X \cdot e_b)\} d  \Sigma. \\
\end{equation}
\end{lemma}
\begin{proof}
For the isometric embedding into the orthogonal complement of $\frac{\partial}{\partial t}$, we have $\breve e_4 =\frac{\partial}{\partial t}$ and 
\[ \langle H_0, \breve e_4 \rangle =0 .\] Hence, $\bar e_4 = \frac{J}{|H|}$ and 
\[ \alpha_{\bar e_3}= \alpha_{\breve e_3} =0. \]
This verifies the assumption of Theorem \ref{eq_comparison_quasilocal_mass}.
Moreover, for the isometric embedding into the orthogonal complement of $\frac{\partial}{\partial t}$, $h_4=0$. As a result of  Theorem \ref{eq_comparison_quasilocal_mass}, we get
\begin{equation} 
   m_{LY}(\Sigma) = \frac{1}{8 \pi} \int\{ [2 det(h_3) - tr h_3 tr h_3' + h_3 \cdot h_3'  ]  Q_{34} +R^{ab}_{\;\;\;\; a4} Q_{b3} -R^{ab}_{\;\;\;\; a3} Q_{b4} \} d  \Sigma .\\
\end{equation}
Finally, we observe that $Q_{ab}=0$, $Q_{a3} =0$ and
\[
\begin{split}
Q_{34}= & (X \cdot e_3)\\
Q_{b4}= & (X \cdot e_b) 
\end{split}
\]
since $\breve e_4 =\frac{\partial}{\partial t}$ and $Q= r\frac{\partial}{\partial r} \wedge \frac{\partial}{\partial t} $. 
\end{proof}
From Lemma \ref{Quasi_local_mass_as_curvature}, we derive the following upper bound for the Liu-Yau quasi-local mass in terms of the curvature tensor of $N$ along $\Sigma$.
\begin{theorem}  \label{upper_bound_curvature_minkowski}
Let $\Sigma$ be a topological sphere in a spacetime $N$. Let $e'_3 =-\frac{H}{|H|}$ and $e'_4 = \frac{J}{|H|}$. Suppose $\alpha_{e'_3}=0$ and ${R^{ab}}_{a4} =0$.  
Finally, assume the second fundamental form in the direction of $e'_3$ is positive definite and the Gauss curvature of the induced metric on $\Sigma$ is positive. We have
\[ 
m_{LY}(\Sigma) \le \frac{1}{8 \pi} \int \{ 2 R_{1212}^+(X \cdot e_3)  -R^{ab}_{\;\;\;\; a3} (X \cdot e_b) \}d  \Sigma.
 \]
 where $R_{1212}^+ = \max\{ R_{1212}, 0\}$.
\end{theorem}
\begin{proof}
We use the Gauss equation 
\[
K=  det (h'_3) - det(h'_4) + R_{1212}
\]
and the Codazzi equations
\[
\begin{split}
\nabla_a (h'_3)_{bc} -\nabla_b (h'_3)_{ac}=&R_{abc3}+ (\alpha_{e'_3})_b (h'_4)_{ac} -(\alpha_{e'_3})_a (h'_4)_{bc}\\
\nabla_a (h'_4)_{bc} -\nabla_b (h'_4)_{ac}=&R_{abc4}+ (\alpha_{e'_3})_b (h'_3)_{ac} -(\alpha_{e'_3})_a (h'_3)_{bc}
\end{split}
\]
 of the surface $\Sigma$ in $N$. By our assumption, $ \alpha_{e'_3} = 0$ and ${R^{ab}}_{a4} =0$. As $tr h_4'=0$, the last Codazzi equation gives
\[   \nabla^c (h'_4)_{bc}  = \nabla^c (h'_4)_{bc}  - \nabla_b tr h_4' = 0.\]
This implies $h'_4=0$ since it is  traceless and symmetric. The Gauss equation simplifies to
\[ K=  det (h'_3) + R_{1212}. \]
We estimate $2 det(h_3) -tr h_3 tr h_3' +h_3 \cdot h_3' $, keeping in mind that  $h_3 $ and $h_3'$ are both positive definite. At each point, we diagonalize $h_3$ and assume
\[ h_3= \left( \begin{array}{cc}
a & 0  \\
0 & b \\ 
\end{array} \right)   \qquad  h'_3= \left( \begin{array}{cc}
a' & c'  \\
c' & b' \\ 
\end{array} \right)  .\] 
We compute 
\[2 det(h_3) -tr h_3 tr h_3' +h_3 \cdot h_3'= 2 ab - a'b -ab' .\]
The Gauss equations of the surface $\Sigma$ in $N$ and $X(\Sigma)$ in $\R^3$ read
\[
K= ab  \quad, \quad K =a'b'- (c')^2 + R_{1212}.
\]
We claim that 
\begin{enumerate}
\item For the points where $R_{1212} \le  0$, we have $2 ab - a'b -ab'\le 0 $ 
\item For the points where $R_{1212} > 0$, we have $2 ab - a'b -ab'\le 2 R_{1212}. $ 
\end{enumerate}
The first claim is easy.  If $R_{1212} \le 0$, then $a'b' \ge K$ and \[ a'b +ab' \ge a'b  +\frac{K^2}{a'b} \ge 2K. \]
For the second case, we have
\[a'b' > K - R_{1212}.\]
Let $C>1 $ be the constant such that 
\[ (K - R_{1212}) C^2 = K.   \]
We have $(Ca')(Cb') \ge  K$ and
\[ \begin{split} 
2 ab - a'b -ab'  \le 2 K -\frac{2K}{C} \le 2 R_{1212}. 
\end{split}
 \]
\end{proof}
From the above upper bound for the Liu-Yau quasi-local mass,  we obtain the following rigidity theorem characterizing the Minkowski spacetime.
\begin{theorem} \label{rigidity_Minkowski}
Let $\Sigma$ be a surface in a spacetime $N$ satisfying the dominant energy condition. Suppose $\Sigma$ bounds a spacelike hypersurface $M$. Let $e'_3 =-\frac{H}{|H|}$ and $e'_4 = \frac{J}{|H|}$. Suppose $\alpha_{e'_3}=0$, ${R^{ab}}_{a4} =0$, $\nabla_b {R^{ab}}_{a3} = 0$, and $R_{1212} \le 0$ on $\Sigma$. Finally, assume the second fundamental form in the direction of $e_3$ is positive definite and the Gauss curvature of $\Sigma$ is positive. Then the domain of dependence of $M$ is isometric to a open set in $\R^{3,1}$.
\end{theorem}
\begin{proof}
We have
\[ 
 (X \cdot e_2) e_2+ (X \cdot e_1) e_1= \frac{1}{2}\nabla  |X|^2.
\]
It follows that
 \[ \int  [ {R^{ab}}_{a3} (X \cdot e_b) ] d  \Sigma = - \frac{1}{2}\int [|X|^2  \nabla_b {R^{ab}}_{a3}]d\Sigma      = 0. \]
We conclude that 
\[ 
m_{LY}(\Sigma) \le \frac{1}{8 \pi} \int  [2R_{1212}^+ (X \cdot e_3)  -{R^{ab}}_{ a3} (X \cdot e_b)] d  \Sigma = 0.
 \] 
 The theorem follows from the positivity and rigidity of the Liu-Yau quasi-local mass \cite{Liu-Yau}. 
\end{proof}
As corollaries of the above theorem, we have the following two statements about time-symmetric initial data $(M,\bar g)$. Let $\bar R_{ijkl}$ and $\bar Ric_{ij}$ be the Riemannian curvature and Ricci curvature tensor of $g$, respectively. First, we get the following upper bound of the Brown-York quasi-local mass $m_{BY}$.
\begin{corollary} \label{upper_bound_Brown-York}
Let $\Sigma$ be a convex surface in a time-symmetric hypersurface $(M,\bar g)$ with positive Gauss curvature. We have the following upper bound for its Brown-York quasi-local mass.
\begin{equation} \label{formulae_brown-York}
\begin{split}
m_{BY}(\Sigma)= &  \frac{1}{8 \pi} \int \{ [2 det(h_3) - tr h_3 tr h_3' + h_3 \cdot h_3'  ]  (X \cdot e_3)- {{\bar R}^{ab}}_{\,\,\,\,\,\,\, a3} (X \cdot e_b)\} d  \Sigma.  \\
\le & \frac{1}{8 \pi} \int \{ 2 \bar R_{1212}^+(X \cdot e_3)d  \Sigma + \frac{1}{16 \pi} \int|X|^2 \nabla_b  {{\bar R}^{ab}}_{\,\,\,\,\,\,\, a3}\} d\Sigma.
\end{split}
 \end{equation}
 where $\bar R_{1212} ^+= \max \{ \bar R_{1212}, 0\}$.
\end{corollary}
\begin{proof}
Consider the static spacetime $N$ with the metric 
\[
g =  - dt^2 + \bar g 
\]
We have
\[  R_{ijkl} = \bar R_{ijkl} \]
and
\[
R_{ijk 0} = 0.
\]
The corollary follows from Lemma  \ref{Quasi_local_mass_as_curvature} and Theorem \ref{upper_bound_curvature_minkowski}.
\end{proof}
We also get the following rigidity theorem:
\begin{theorem}\label{rigidity_R3}
Let $(M,\bar g)$  be a 3-manifold with boundary $\Sigma$. Assume that  the scalar curvature of $\bar g$ is non-negative and $\Sigma$ is a convex 2-sphere with positive Gauss curvature. If $\nabla_b  {{\bar R}^{ab}}_{\,\,\,\,\,\,\, a3} = 0$ and $\bar R_{1212} \le 0$ on $\Sigma$. Then $\bar g$ is the flat metric. 
\end{theorem}
\begin{proof}
It follows from Corollary \ref{upper_bound_Brown-York} that 
\[
m_{BY} (\Sigma) \le 0. 
\] 
The theorem follows from the the positivity and rigidity of the Brown-York mass. 
\end{proof}
For asymptotically flat initial data sets, it is known that the limit of the Brown-York-Liu-Yau quasi-local mass of the coordinate spheres recovers the ADM mass of the initial data \cite{fst}.
In the general relativity literature (see Ashtekar--Hansen \cite{AH}, Chru\'{s}ciel \cite{Chrusciel} and Schoen \cite{Schoen}),
it is also known that the ADM mass can be computed using the Ricci curvature of the induced metric of the initial data. As a corollary of Lemma 4.1, we obtain a simple proof that the limit of the Brown-York-Liu-Yau quasi-local mass coincides with the ADM mass via the Ricci curvature (see also the earlier proof by Miao--Tam--Xie in  \cite{Miao-Tam-Xie}).
\begin{definition}
$(M^3, \bar g)$ is an {\it asymptotically flat} manifold of order $\tau$ if, outside a compact set,
$M^3$ is diffeomorphic to  $ \R^3  \setminus \{ | x | \le r_0 \}$ for some $r_0>0$
and under the diffeomorphism, we have
\[
\bar g_{ij} - \delta_{ij} = O ( | x |^{- \tau} ) ,  \partial \bar g_{ij} = O ( |x|^{-1-\tau}), \  \partial^2 \bar g_{ij} = O (|x|^{-2 - \tau} )
\]
for some $\tau>\frac {1}{2}$. Here $ \partial $ denotes the partial differentiation on $ \R^3$. 
\end{definition}

\begin{theorem}\label{limit_rewrite}
Suppose we have an asymptotically manifold of order $\alpha> \frac{1}{2}$ and $\Sigma_r$ be the coordinate spheres of the asymptotically flat coordinates. We have
\begin{equation}
\lim_{r \to \infty} \int [ H_0 - H + (\bar Ric - \frac{1}{2} \bar R  \bar g)(X, e_3) ] d \Sigma_r = 0.
\end{equation}
Here $X$ denote the position vector of $\R^3$ which is identified with a vector field along the surface $\Sigma_r$ on the initial data set via the canonical gauge of the isometric embedding.
\end{theorem}
\begin{proof}
Using \eqref{formulae_brown-York}, we have
\[
\begin{split}
    &\int ( H_0 - H )d\Sigma_r \\
 = &\int \{ [2 det(h_3) - tr h_3 tr h_3' + h_3 \cdot h_3'  ]  (X \cdot e_3)-\bar Ric (e_3 , e_a) X \cdot e_a \} d  \Sigma_r\\
 = & \int \{ det(h_3 -h_3')  (X \cdot e_3) + [det(h_3) - det(h_3') ] (X \cdot e_3)  -\bar Ric (e_3 , e_a) X \cdot e_a \} d  \Sigma_r\\
 = & \int \{ det(h_3 -h_3')  (X \cdot e_3) + [ \frac{\bar R}{2} - \bar Ric(e_3,e_3) ] (X \cdot e_3)  -\bar Ric (e_3 , e_a) X \cdot e_a \}d  \Sigma_r \\
 = & \int \{ det(h_3 -h_3') (X \cdot e_3) + [ \frac{1}{2} \bar R \bar g - \bar Ric ] (X , e_3)  \} d  \Sigma_r 
 \end{split}
\]
For  an asymptotically flat initial data set of order $\tau> \frac{1}{2}$, it is straight-forward to check that $X \cdot e_3 = O(r)$ and 
\[
det(h_3 -h_3')   = O(r^{-2 \tau -2}).
\]
\end{proof}
\section{Quasi-local mass with reference to the anti-de Sitter spacetime}
The anti-de Sitter spacetime,
\[  -(1+r^2) dt^2 + \frac{dr^2}{1+r^2} + r^2 dS^2, \]
also admits the Killing-Yano 2-form $Q= r \frac{\partial}{\partial r}\wedge  \frac{\partial}{\partial t}$. In the following, we derive the analogue of Theorem \ref{Minkowski_compare}
 with respect to the anti-de Sitter spacetime.
 \begin{theorem} \label{anti_de_Sitter_comparison}
Given a  spacelike 2-surface $\Sigma$  in $N$ and a frame $\{ e'_3, e'_4\}$ of the normal bundle and let $X$ be an isometric embedding of $\Sigma$ into the Anti-de Sitter spacetime. Suppose there is a frame $\{ e_3, e_4\}$ of the normal bundle of $X(\Sigma)$ such that
\[   \alpha_{e'_3}=\alpha_{e_3}.\]
Then we have
\begin{equation} \label{eq_comparison_ads_2}
\begin{split}
   & \int [-\langle \frac{\partial}{\partial t} , e_4 \rangle(tr h_3 -tr h'_3) +\langle \frac{\partial}{\partial t} , e_3 \rangle(tr h_4 -tr h'_4)] d  \Sigma \\
=&  \int [2 det(h_3) -2det(h_4) - tr h_3 tr h_3' +h_3 \cdot h_3' +tr h_4 tr h_4' - h_4 \cdot h_4' ] Q_{34}d  \Sigma \\
 &+\int  [ R^{ab}_{\;\;\;\; a4} Q_{b3}-R^{ab}_{\;\;\;\; a3}Q_{b4} - Q_{bc} \sigma^{cd}[(h_{3})_{da}  h_4^{'ab}-(h_{4})_{da}  h_3^{'ab} ] ] d  \Sigma \\
\end{split}  
\end{equation}
where $R$ is the curvature tensor for the spacetime $N$. \end{theorem}
\begin{proof}
The proof is the same as that of  Theorem \ref{Minkowski_compare}. While the anti-de Sitter space is not flat, it is a space form and thus, curvature tensor for the reference space does not show up in the formula.
\end{proof}
We relate the left hand side of equation \eqref{eq_comparison_ads_2} to the quasi-local mass with the reference in the anti-de Sitter spacetime when the frame $\{ e'_3, e'_4\}$  and $\{ e_3, e_4\}$ are the canonical gauge corresponding to a pair of isometric embedding $X$ of  $\Sigma$ into the anti-de Sitter spacetime and the Killing vector field $T_0=\frac{\partial}{\partial t}$. 

The Killing vector field $T_0 $ generates a one-parameter family of isometries $\phi_t$ of the anti-de Sitter spacetime. Let $C$ be the image of $\Sigma$ under the one-parameter family $\phi_t$. The intersection of $C$ with the static slice $t=0$ is $\widehat \Sigma$. By a slight abuse of terminology, we refer to $\widehat \Sigma$ as the projection of $\Sigma$. Let $\breve e_3$ be the outward unit normal of $\widehat \Sigma$ in the static slice $t=0$. Consider the pushforward of $\breve e_3$ by the one-parameter family  $\phi_t$, which is denoted by $\breve e_3$ again. Let $\breve e_4$ be the future directed unit normal of $\Sigma$ normal to $\breve e_3$ and extend it along $C$ in the same manner. Let $\{ \bar e_3, \bar e_4\}$ be the unique frame of the normal bundle of $\Sigma$ in $N$ such that 
\[  \langle H ,  \bar e_4 \rangle =  \langle H_0 ,  \breve e_4 \rangle. \]
The quasi-local energy of $\Sigma$ with respect to the pair $(X,T_0)$ is 
\[  E(\Sigma,X,T_0)= \frac{1}{8 \pi} \int  [\langle H_0 ,  \breve e_3 \rangle \langle \frac{\partial}{\partial t}, \breve e_4 \rangle+ \alpha_{\breve e_3}(T_0^T)-\langle H ,  \bar e_3 \rangle  \langle \frac{\partial}{\partial t}, \breve e_4 \rangle - \alpha_{\bar e_3}(T_0^T)]d \Sigma. \]
Assume again that $ \alpha_{\bar e_3}= \alpha_{\breve e_3}$ and use $tr h_4 =tr h'_4$, we have
\begin{equation} \label{eq_comparison_mass_ads_2}
 E(\Sigma,X,T_0)=    \frac{1}{8 \pi}  \int [ -\langle \frac{\partial}{\partial t} ,\breve e_4 \rangle(tr h_3 -tr h'_3)] d  \Sigma \\
\end{equation}
To summarize, we have proved the following:
\begin{theorem} \label{comparison_quasilocal_mass_ads}
Given a surface $\Sigma$ in the spacetime $N$, suppose we have an isometric embedding $X$ of $\Sigma$ into the anti-de Sitter spacetime and  the Killing vector field $T_0=\frac{\partial}{\partial t}$
such that $ \alpha_{\bar e_3}= \alpha_{\breve e_3}$, then we have
\begin{equation} 
\begin{split}
   E(\Sigma,X,T_0) 
=& \frac{1}{8 \pi} \int [2 det(h_3) -2det(h_4) -tr h_3 tr h_3' +h_3 \cdot h_3' +tr h_4 tr h_4' - h_4 \cdot h_4' ] Q_{34}d  \Sigma \\
 &  +\frac{1}{8 \pi} \int  \{ R^{ab}_{\;\;\;\; a4} Q_{b3}-R^{ab}_{\;\;\;\; a3}Q_{b4} - Q_{bc} \sigma^{cd}[(h_{3})_{da}  h_4^{'ab}-(h_{4})_{da}  h_3^{'ab} ] \} d  \Sigma \\
\end{split}  
\end{equation}
\end{theorem}

In particular, we can rewrite the quasi-local mass with reference in the $t=0$ slice of the anti-de Sitter spacetime.
\begin{corollary} \label{Quasi_local_mass_as_curvature_ads}
Suppose $\alpha_H = 0$.  Let $X$ be the isometric embedding of the surface into the $t=0$ slice in the anti-de Sitter spacetime. We have
\begin{equation} \label{mass_hyperbolic}
   E(\Sigma,X,\frac{\partial}{\partial t}) = \frac{1}{8 \pi} \int \{ [2 det(h_3) - tr h_3 tr h_3' + h_3 \cdot h_3'  ]  (r  \frac{\partial}{\partial r}\cdot e_3)-R^{ab}_{\;\;\;\; a3} (r \frac{\partial}{\partial r} \cdot e_b) \}d  \Sigma. \\
\end{equation}
\end{corollary}
\begin{proof}
For the  isometric embedding of the surface into the $t=0$ slice of the anti-de Sitter spacetime, we have $h_4=0$ and $Q_{ab} = Q_{a3} =0$. Hence
\eqref{mass_hyperbolic} follows from Theorem \ref{comparison_quasilocal_mass_ads}. 
\end{proof}
We get the following rigidity theorem:
\begin{theorem}\label{rigidity_ads}
Let $(M,\bar g)$  be a 3-manifold with boundary. Assume that  the scalar curvature $\bar R(g)$ satisfies $\bar R(g)\ge -6$, the boundary is convex and the Gauss curvature of the induced metric is bounded from below by  $-1$. Let $\bar R_{ijkl}$ be the curvature tensor of $\bar g$. If  $\nabla_b  {{\bar R}^{ab}}_{\,\,\,\,\,\,\, a3}  = 0$ and $\bar R_{1212} \le -1$ on the boundary. Then $g$ is the hyperbolic metric. 
\end{theorem}
\begin{proof}
We pick an isometric embedding of $\Sigma$ into the hyperbolic space such that \[  r \frac{\partial}{\partial r}    \cdot e_3>0. \]
In view of the form the hyperbolic metric $\frac{dr^2}{1+r^2} + r^2 dS^2$ on the $t=0$ slice, it is easy to check that $r \frac{\partial}{\partial r}$ is a gradient vector field with the 
potential $\frac{1}{2} r^2+\frac{1}{4} r^4$. 
In particular, $ r \frac{\partial}{\partial r}\cdot e_b=\nabla_b(\frac{1}{2} r^2+\frac{1}{4} r^4)$ on $X(\Sigma)$. 

Let $N$ be the spacetime with metric
\[
g = -dt^2 + \bar g.
\]
We have
\[
R_{ijkl} =\bar R_{ijkl}.
\]
We apply Corollary \ref{Quasi_local_mass_as_curvature_ads} to express the quasi-local energy. By our assumption, 
\[  \int  R^{ab}_{\;\;\;\; a3} (  r \frac{\partial}{\partial r}  \cdot e_b) d  \Sigma =  0.\]
As a result, \eqref{mass_hyperbolic} implies
\[   E(\Sigma,X,\frac{\partial}{\partial t}) = \frac{1}{8 \pi} \int [2 det(h_3) - tr h_3 tr h_3' + h_3 \cdot h_3'  ]  (r  \frac{\partial}{\partial r}\cdot e_3) d  \Sigma. \\
\]
The Gauss equations read
\[
\begin{split}
K= &\bar R_{1212} + det (h_3')\\
K=& -1 + det(h_3).
\end{split}
\]
We estimate \[2 det(h_3) - tr h_3 tr h_3' + h_3 \cdot h_3'  \]
as in the proof of Theorem \ref{upper_bound_curvature_minkowski} and use  the assumption that $\bar R_{1212} \le -1$. We get
\begin{equation} 
   E(\Sigma,X,\frac{\partial}{\partial t}) \le 0.
     \end{equation}
     The theorem now follows from the positive mass theorem of \cite{Shi-Tam}.
\end{proof}
For an asymptotically hyperbolic manifold, we can use Corollary 5.3 to express the limit of the quasi-local mass with reference in the hyperbolic space in terms of the limit of Ricci curvature similar to Theorem 4.6. 
This gives a new proof of the results proved by Herzlich in \cite{Herzlich} and by Miao, Tam and Xie in \cite{Miao-Tam-Xie}. Let 
\[ g_0 = \frac{dr^2}{r^2+1} + r^2 dS^2
\] be the hyperbolic metric of the hyperbolic space $ \HH^3$ and $V=\sqrt{r^2+1}$ be the static potential.
\begin{definition}
$(M^3,\bar g)$ is an {\it asymptotically hyperbolic} manifold of order $\tau$  if, outside a compact set,
$M^3$ is diffeomorphic to  $ \HH^3  \setminus \{ | x | \le r_0 \}$ for some $r_0>0$.
Under the diffeomorphism, we have 
\[
|\bar g - g_0| = O ( | x |^{- \tau} ) , |\partial (\bar g - g_0) |  = O ( |x|^{-1-\tau}), \  \partial^2 (\bar g - g_0) = O (|x|^{-2 - \tau} )
\]
for some $\tau>\frac {3}{2}$. Here $ \partial $ denotes the partial differentiation with respect to the coordinate system of $\HH^3$ and the norm is measured with respect to $g_0$.
\end{definition}
\begin{theorem}\label{limit_rewrite_hyperbolic}
Suppose we have an asymptotically flat initial data set of order $\alpha> \frac{3}{2}$ and $\Sigma_r$ be the coordinate spheres of the asymptotically flat coordinates. We have
\begin{equation}
\lim_{r \to \infty} \int \{V ( H_0 - H) + (\bar Ric - \frac{1}{2}  (\bar R+2)\bar g)(X, e_3) \} d \Sigma_r = 0.
\end{equation}
\end{theorem}
\begin{proof}
The proof is the same as Theorem 4.6 using Corollary 5.3 instead of Corollary \ref{upper_bound_Brown-York}. The resulting formula has $\bar R+2$ instead of $\bar R$ due to the curvature of the hyperbolic space when applying the Gauss equation to the image of the isometric embedding.
\end{proof}

\section{Quasi-local mass with reference to the Schwarzschild spacetime}
The Schwarzschild spacetime spacetime,
\[  - (1-\frac{2M}{r}) dt^2 + \frac{dr^2}{1-\frac{2M}{r}} + r^2 dS^2, \]
also admits the Killing-Yano 2-form $Q= r \frac{\partial}{\partial r}\wedge  \frac{\partial}{\partial t}$. In the following, we derive the analogue of Theorem \ref{Minkowski_compare}
 with respect to the Schwarzschild spacetime.
 \begin{theorem} \label{Sch_comparison}
Given a  spacelike 2-surface $\Sigma$  in $N$ and a frame $\{ e'_3, e'_4\}$ of the normal bundle and let $X$ be an isometric embedding of $\Sigma$ into the Schwarzschild spacetime. Suppose there is a frame $\{ e_3, e_4\}$ of the normal bundle of $X(\Sigma)$ such that
\[   \alpha_{e'_3}=\alpha_{e_3}.\]
Then we have
\begin{equation} \label{eq_comparison_sch_1}
\begin{split}
   & \int \{-\langle \frac{\partial}{\partial t} , e_4 \rangle(tr h_3 -tr h'_3) +\langle \frac{\partial}{\partial t} , e_3 \rangle(tr h_4 -tr h'_4) \}d  \Sigma \\
=&  \int \{ [2 det(h_3) -2det(h_4) - tr h_3 tr h_3' +h_3 \cdot h_3' +tr h_4 tr h_4' - h_4 \cdot h_4' ] Q_{34}\}d  \Sigma \\
 &+\int  \{ (R^{ab}_{\;\;\;\; a4} -{R_s^{ab}}_ {a4} )Q_{b3}-(R^{ab}_{\;\;\;\; a3}- {R_s^{ab}}_{a3} )Q_{b4} - Q_{bc} \sigma^{cd}[(h_{3})_{da}  h_4^{'ab}-(h_{4})_{da}  h_3^{'ab} ] \} d  \Sigma \\
\end{split}  
\end{equation}
where $R$ and $R_s$ the curvature tensor for the spacetime $N$ and the Schwarzschild spacetime, respectively. \end{theorem}
\begin{proof}
The proof is the same as that of  Theorem \ref{Minkowski_compare}. The corresponding curvature terms appear when we apply the Codazzi equation of the surface in the Schwarzschild spacetime.
\end{proof}
We relate the left hand side of equation \eqref{eq_comparison_sch_1} to the quasi-local mass with the reference in the Schwarzschild spacetime when the frame $\{ e'_3, e'_4\}$  and $\{ e_3, e_4\}$ are the canonical gauge corresponding to a pair of isometric embedding $X$ of  $\Sigma$ into the Schwarzschild spacetime and the Killing vector field $T_0=\frac{\partial}{\partial t}$. 

The Killing vector field $T_0 $ generates a one-parameter family of isometries $\phi_t$ of the Schwarzschild spacetime. Let $C$ be the image of $\Sigma$ under the one-parameter family $\phi_t$. The intersection of $C$ with the static slice $t=0$ is $\widehat \Sigma$. By a slight abuse of terminology, we refer to $\widehat \Sigma$ as the projection of $\Sigma$. Let $\breve e_3$ be the outward unit normal of $\widehat \Sigma$ in the static slice $t=0$. Consider the pushforward of $\breve e_3$ by the one-parameter family  $\phi_t$, which is denoted by $\breve e_3$ again. Let $\breve e_4$ be the future directed unit normal of $\Sigma$ normal to $\breve e_3$ and extend it along $C$ in the same manner. Let $\{ \bar e_3, \bar e_4\}$ be the unique frame of the normal bundle of $\Sigma$ in $N$ such that 
\[  \langle H ,  \bar e_4 \rangle =  \langle H_0 ,  \breve e_4 \rangle. \]
The quasi-local energy of $\Sigma$ with respect to the pair $(X,T_0)$ is 
\[  E(\Sigma,X,T_0)= \frac{1}{8 \pi} \int  [\langle H_0 ,  \breve e_3 \rangle \langle \frac{\partial}{\partial t}, \breve e_4 \rangle+ \alpha_{\breve e_3}(T_0^T)-\langle H ,  \bar e_3 \rangle  \langle \frac{\partial}{\partial t}, \breve e_4 \rangle - \alpha_{\bar e_3}(T_0^T)] d\Sigma. \]
Assume again that $ \alpha_{\bar e_3}= \alpha_{\breve e_3}$ and use $tr h_4 =tr h'_4$, we have
\begin{equation} \label{eq_comparison_mass_sch_2}
 E(\Sigma,X,T_0)=    \frac{1}{8 \pi}  \int [-\langle \frac{\partial}{\partial t} , \breve e_4 \rangle(tr h_3 -tr h'_3) ]d  \Sigma \\
\end{equation}
To summarize, we have proved the following:
\begin{theorem} \label{comparison_quasilocal_mass_sch}
Given a surface $\Sigma$ in the spacetime $N$, suppose we have an isometric embedding $X$ of $\Sigma$ into the Schwarzschild spacetime and the Killing vector field $T_0=\frac{\partial}{\partial t}$
such that $ \alpha_{\bar e_3}= \alpha_{\breve e_3}$, then we have
\begin{equation} 
\begin{split}
   & E(\Sigma,X,T_0) \\
=& \frac{1}{8 \pi} \int [2 det(h_3) -2det(h_4) -tr h_3 tr h_3' +h_3 \cdot h_3' +tr h_4 tr h_4' - h_4 \cdot h_4' ] Q_{34}d  \Sigma \\
   &+\frac{1}{8 \pi} \int  \{ (R^{ab}_{\;\;\;\; a4} -{R_s^{ab}}_ {a4} )Q_{b3}-(R^{ab}_{\;\;\;\; a3}- {R_s^{ab}}_{a3} )Q_{b4} - Q_{bc} \sigma^{cd}[(h_{3})_{da}  h_4^{'ab}-(h_{4})_{da}  h_3^{'ab} ] \} d  \Sigma \\
\end{split}  
\end{equation}
\end{theorem}

\end{document}